\documentclass[a4paper]{article}

\usepackage{lipsum}
\usepackage{amsmath,amsthm,amssymb}
\usepackage{graphicx}
\usepackage{hhline,url}
\usepackage{authblk}
\usepackage{harvard}
\citationmode{abbr}

\newtheorem{thm}{Theorem}[section]
\newtheorem{lem}[thm]{Lemma}
\newtheorem{remark}[thm]{Remark}
\newenvironment{rem}{\begin{remark}\rm}{\end{remark}}
\newtheorem{algor}[thm]{Method}
\newenvironment{algo}{\begin{algor}\rm}{\end{algor}}
\newtheorem{prop}[thm]{Proposition}
\newtheorem{Definition}[thm]{Definition}
\newenvironment{dfn}{\begin{Definition}\rm}{\end{Definition}}
\newtheorem{Corollary}[thm]{Corollary}

\newtheorem{Example}{Example}[section]

\newcommand{\acite}{\citeasnoun}

\renewcommand{\phi}{\varphi}

\newcommand{\R}{\mathbb{R}}
\newcommand{\dd}{\,\mathrm{d}}
\newcommand{\ve}{\varepsilon}

\newcommand{\F}{\mathcal{F}}
\newcommand{\E}[1]{\mathrm{E}\left[#1\right]}
\newcommand{\PP}{\mathcal{P}}
\newcommand{\X}{\mathcal{X}}
\newcommand{\bm}[1]{{\mbox{\boldmath $#1$}}}
\DeclareMathOperator{\im}{Im}
\DeclareMathOperator{\cv}{conv}

\DeclareMathOperator{\ri}{ri}
\DeclareMathOperator{\supp}{supp}
\DeclareMathOperator{\aff}{aff}
\newcommand{\lmid}{\,\middle|\,}
\newcommand{\B}{\mathcal{B}}

\makeatletter

 \@addtoreset{equation}{section}
\makeatother

\title{Monte Carlo Cubature Construction}

\author[1]{Satoshi Hayakawa\footnote{satoshi\_hayakawa@mist.i.u-tokyo.ac.jp}}

\affil[1]{Graduate School of Information Science and Technology,
The University of Tokyo}

\date{}

\begin{document}

\maketitle

\begin{abstract}
	In numerical integration,
	cubature methods are effective,
	especially when the integrands
	can be well-approximated by known test functions, such as polynomials.
	However, the construction of cubature formulas has not generally been known,
	and existing examples only represent the particular domains of integrands, such as hypercubes and spheres.
	In this study, we show that cubature formulas can be constructed for probability measures
	provided that
	we have an i.i.d. sampler from the measure and the mean values of given test functions.
	Moreover, the proposed method also works as a means of data compression,
	even if sufficient prior information of the measure is not available.
\end{abstract}

\section{Introduction}
	In this study, we consider a method to obtain a
	good numerical integration formula over a probability measure.
	In other words,
	given a random variable $X$ taking values in some abstract space,
	we want an accurate approximation of
	the expectation $\E{f(X)}$ for each integrand $f$.
	
	If there are no conditions for the class of integrands
	(except that $f(X)$ is integrable), it is expected that
	the Monte Carlo method \cite{met49,ham64} should be the best way.
	For i.i.d. copies $X_1,\ldots, X_N, \ldots$ of $X$,
	the sample mean of
	$\frac1N\left(f(X_1)+\cdots+f(X_N)\right)$ is an unbiased estimator of
	$\E{f(X)}$ and, according to the law of large numbers, converges to the expectation with probability $1$.
	If we assume $\E{|f(X)|^2}<\infty$,
	the standard variation of Monte Carlo integration is $\mathrm{O}(\frac1{\sqrt{N}})$.
	Therefore, we can regard this as a numerical integration method
	with error $\mathrm{O}(\frac1{\sqrt{N}})$.
	
	However, in real applications,
	we often assume that the integrand $f$ has good properties, such as
	smoothness, asymptotic behavior, and the decay of coefficients when expanded by a certain basis.
	Indeed, these itemized properties are not independent of each other;
	however, they are possibly suited for different numerical integration methods.
	For instance, if $X$ is a uniform distribution and integrands have a certain smoothness over a hypercube $[0, 1]^s$,
	then the quasi Monte Carlo (QMC) sampling works better than the standard Monte Carlo method \cite{dic13}.
	The points used in QMC are usually taken in a deterministic way,
	but random QMC construction methods have recently been considered by \citeasnoun{hir16} and \citeasnoun{bar16}.
	In particular, the authors of these two papers exploit determinantal point processes,
	which physically model the distribution of repulsing particles.
	
	If $X$ takes values in some Euclidean space,
	and the integrand $f$ can be well-approximated by polynomials,
	cubature (or quadrature) formulas are useful \cite{str71,coo93,EDT}.
	The simplest one is the approximation
	\[
		\int_0^1f(x)\dd x\simeq \frac16f(0)+\frac23f\left(\frac12\right)+\frac16f(1),
	\]
	which is called Simpson's rule.
	This `$\simeq$' is replaced by `$=$' if $f$ is a polynomial with three degrees at most.
	In general, a cubature formula of degree $t$ is composed of points $x_1,\ldots, x_n$
	and weights $w_1,\ldots,w_n>0$ satisfying
	\[
		\E{f(X)}=\sum_{i=1}^nw_if(x_i)
	\]
	for any polynomial $f$ with degrees of $t$ at most.
	The existence of this formula is assured by Tchakaloff's theorem (Theorem \ref{thm:tchakaloff}),
	but the proof is not constructive and, except for concrete cases, construction is currently unknown.
	It is not necessary for us to choose polynomials as test functions,
	so we can consider the generalized cubature formulas
	on spaces and test functions that are non-necessarily Euclidean and polynomial.
	Indeed, such formulas are worthy of consideration; moreover,
	cubature formulas exists for Wiener space \cite{lyo04}.
	As QMC methods and other widely used numerical integration formulas, such as DE (double exponential) methods \cite{tak74},
	are limited to the integration of functions with a certain smoothness and domain, Monte Carlo integration is essentially the only available method for general settings.
	Accordingly, although it is difficult to choose proper test functions, it is valuable to consider general cubature formulas.
	
	In this paper,
	we propose an algorithm to construct generalized cubature formulas that satisfy Tchakaloff's bound
	by subsampling points from i.i.d. samples $X_1, X_2, \ldots$ of $X$.
	Therefore, the proposed method is referred to as Monte Carlo cubature construction.
	This construction is realized by finding a basic feasible solution of a linear programming (LP) problem,
	providing weights of a cubature formula.
	Although the proposed method is simple,
	it may fail if pathological samples $X_1, X_2, \ldots$ appear.
	However, we show that the probability of such cases is zero.
	In other words,
	we show that one can construct a general cubature formula with probability $1$ so long as the i.i.d. sequence
	$X_1, X_2, \ldots$ of $X$ and the expectation of test functions are both given.
	This is the main theoretical result presented by Theorem \ref{thm:main},
	which has been proven by techniques in discrete geometry.
	
	In some cases,
	cubature formulas might exist with fewer points than
	the upper bound given by Tchakaloff's theorem
	($n = t$ in the aforementioned one-dimensional example).
	For example,
	such cubature formulas have been found
	for particular probability measures and on (particular-dimensional hyper) cubes, triangles, circles (spheres), and balls \cite{coo93,EDT}.
	However, our construction is rather general
	as we have almost no restrictions on the probability measure, domain, and test functions.

	Furthermore,
	we also demonstrate that one can construct an approximate cubature formula
	in the absence of knowing test-function expectations.
	This construction of an approximate cubature formula is based on the usual Monte Carlo integration,
	but we can compress the data via construction of an cubature formula.
	
	The rest of this paper is organized as follows:
	in Section \ref{sec:2},
	we briefly review cubature formulas, a subsampling technique, and theorems in discrete geometry;
	in Section \ref{sec:3},
	we present our main results;
	we give simple numerical experiments in Section \ref{sec:4}
	to estimate the time complexity of our method; finally,
	we conclude the paper in Section \ref{sec:5}.


\section{Background}\label{sec:2}

In this section, we outline related studies, including the classic formulation of cubature formulas (and its generalization)
and theorems in discrete geometry that are used in the proofs of our results.

	\subsection{Cubature formula}
	
	Herein, we let $\Omega\subset\R^m$ be a Borel set
	and we consider a probability measure $\mu$ on $\Omega$ (with Borel $\sigma$-field).
	Moreover, for an integer $t\ge0$, we denote by $\PP_t(\Omega)$ the set of all real polynomials over $\Omega$ with degree $t$ at most.
	A {\it cubature formula of degree $t$}
	is a set of points $x_1,\ldots,x_n\in\Omega$ and positive weights $w_1,\ldots,w_n$ such that
	\begin{align}
		\int_\Omega f(x)\dd\mu(x)
		=\sum_{j=1}^n w_jf(x_j)
		\label{eq:cubature-poly}
	\end{align}
	holds for any $f\in\PP_t(\Omega)$.
	We have assumed that each $f\in\PP_t(\Omega)$ is integrable with $\mu$,
	which is satisfied, for instance, if $\Omega$ is bounded or $\mu$ has a rapidly decreasing density.
	Once such a formula is obtained
	$\sum_{j=1}^n w_jf(x_j)$ is a good approximation of $\int_\Omega f(x)\dd\mu(x)$
	for a function $f$ that can be well-approximated by polynomials.
	In this sense, a cubature formula works as roughly compressed data
	of $\mu$; moreover, we call $\sum_{j=1}^n w_j\delta_{x_j}$ a {\it
	cubature}.

	The existence of such a formula is assured by \acite{tch57} and \acite{bay06},
	as outlined below.
	\begin{thm}{\rm(Tchakaloff's theorem)}\label{thm:tchakaloff}
		If every function in $\PP_t(\Omega)$ is integrable with $\mu$;
		i.e.,
		\[
			\int_\Omega|f(x)|\dd\mu(x)<\infty, \qquad f\in\PP_t(\Omega)
		\]
		holds,
		then there exists an integer $1\le n\le \dim\PP_t(\Omega)$,
		points $x_1, \ldots, x_n\in \supp\mu\ (\subset \Omega)$, and weights $w_1,\ldots,w_n>0$
		that can be used to construct a cubature formula of degree $t$ over $(\Omega, \mu)$.
	\end{thm}

	Note that $\supp\mu$ is the closed set defined by
	$\supp\mu:=\bigcap\{\Omega\setminus O\mid \text{$O$ is open}, \mu(O)=0\}$.
	Since a nonzero constant function belongs to $\PP_t(\Omega)$, then
	$w_1+\cdots+w_n$ must be equal to $1$.
	
	This theorem can be understood in the context of discrete geometry.
	Indeed, if we define a vector-valued function $\bm\phi:\R^m\to\R^{\dim\PP_t(\Omega)}$
	with a monomial in each component,
	then constructing a cubature formula is equivalent to
	finding a $\bm{y}_1,\ldots, \bm{y}_n\in\im\bm\phi|_{\supp\mu}$ such that
	the convex hull of $\{\bm{y}_1,\ldots,\bm{y}_n\}$ contains the point
	$\int_\Omega\bm\phi(x)\dd\mu(x)\in\R^{\dim\PP_t(\Omega)}$.
	For further arguments from this viewpoint, see
	Section \ref{sec:discrete} or \acite{bay06}.
	
	Tchakaloff's theorem only gives an upper bound of the number of points
	used in a cubature formula.
	There exist, several lower bounds for the number of sampled points $n$.
	The Fisher-type bound is well-known \cite{str71}:
	\begin{thm}{\rm (Fisher-type bound)}
		For any cubature formula of degree $t$ over the measure space
		$(\Omega, \mu)$,
		the number of sample points $n$ satisfies
		$n\ge \dim\PP_{\lfloor t/2\rfloor}(\Omega)$.
	\end{thm}
	However, this study is concerned with finding at least one
	cubature formula for general settings;
	in other words, minimizing the number of sample points is not the objective of this paper.
	See \acite{EDT} for further general theories and examples
	of the cubature.
	
	\subsection{Carath\'{e}odory-Tchakaloff subsampling}
	
	Constructing a cubature formula
	is generally not an easy task.
	However, since the case $\mu$ is a product measure of
	low-dimensional measures (typically one-dimensional),
	we can easily construct a grid-type cubature formula.
	Although the proposition outlined below is on the product measures of
	the same low-dimensional measure,
	it can be generalized for
	the product measure of different low-dimensional measures.
	\begin{prop}\label{prop:classic-cubature}
		Let points $x_1, \ldots, x_n\in \Omega$ and weights $w_1,\ldots, w_n>0$
		make a cubature of degree $t$.
		on $(\Omega, \mu)$.
		Then, on the $k$-fold product measure space
		$(\Omega^{\otimes k}, \mu^{\otimes k})$,
		\[
			\sum_{(j_1,\ldots,j_k)\in\{1,\ldots,n\}^k}
			w_{j_1}\cdots w_{j_k} \delta_{(x_{j_1},\ldots,x_{j_k})}
		\]
		is a cubature of degree $t$.
	\end{prop}
	This construction uses $n^k$ points
	and this is likely to be much larger than the Tchakaloff upper bound
	$\dim \PP_t(\Omega^{\otimes k})$.
	We can reduce the number of points used in such a cubature formula
	by using a LP problem.
	A detailed study on this direction (with generalization)
	is given by \acite{tch15}.
	To explain the idea, the situation is generalized.	
	Instead of polynomials,
	we can naturally take any sequence of basis functions
	to generalize the definition of a cubature formula given in the previous section.
	Indeed, if we know the class of functions that require integration,
	there might be more optimal test functions to use than polynomials.
	Accordingly, the definition outlined below is appropriate.

	\begin{dfn}\label{def:gen-cubature}
		Given a random variable $X$ taking values in a measurable space $\X$
		and $d$ test functions $\phi_1, \ldots, \phi_d: \X\to\R$
		that satisfy the integrability $\E{|\phi_i(X)|}<\infty$ for each $i=1,\ldots, d$,
		a {\it cubature} with respect to $X$ and $\phi_1,\ldots,\phi_d$
		can be described as a set of points $x_1, \ldots, x_n\in\X$ and positive weights $w_1, \ldots, w_n$
		that satisfies
		\begin{align}
			\E{\phi_i(X)}=\sum_{j=1}^n w_j\phi_i(x_j),\qquad
			i=1,\ldots, d.
			\label{eq:cubature-1}
		\end{align}
	\end{dfn}
	If we write $\bm\phi = (\phi_1, \ldots, \phi_d)^\top: \X\to\R^d$, (\ref{eq:cubature-1}) can be equivalently rewritten as
	\begin{align}
	\E{\bm\phi(X)}=\sum_{j=1}^nw_j\bm\phi(x_j)
	\end{align}
	
	For this generalization of the cubature formula,
	we have the generalized Tchakaloff's theorem \cite{bay06,EDT}.
	\begin{thm}\label{thm:gen-tchakaloff}
	{\rm(Generalized Tchakaloff's theorem)}
		Under the same setting as Definition \ref{def:gen-cubature},
		there exists a cubature formula with $n\le d+1$.
		Moreover, if there exists a vector $c\in\R^d$ such that
		$c^\top\bm\phi(X)$ is essentially a nonzero constant,
		there exists a cubature formula satisfying $n\le d$ and
		$w_1+\cdots+w_n=1$.
	\end{thm}
	
	We can take $x_1,\ldots,x_n\in\X$ so as to
	meet $\bm\phi(x_1),\ldots,\bm\phi(x_n)
	\in\supp\mathop\mathrm{P}_{\bm\phi(X)}$, where $\mathop\mathrm{P}_{\bm\phi(X)}$ is the distribution of $\bm\phi(X)$
	over $\R^d$,
	which corresponds to the condition
	$x_1,\ldots,x_n\in\supp\mu$ in Theorem \ref{thm:tchakaloff}.
	This theorem is essentially a direct consequence of
	Carath\'{e}odory's theorem (Theorem \ref{thm:car}).
	
	\begin{rem}\label{rem:rigorous}
		We suppose that there is a probability space $(\Omega, \F_\Omega, \mathrm{P})$
		 $\X$ has a natural $\sigma$-field $\F_\X$,
		and that the random variable $X$ is a measurable map from $(\Omega, \F_\Omega)$ to $(\X, \F_\X)$.
		In addition, we also assume that the test function $\bm\phi$ is a measurable map from $(\X, \F_\X)$
		to $(\R^d, \B(\R^d))$.
		Accordingly, $\bm\phi(X)$ is a random variable on $\R^d$
		(measurable map from $(\Omega, \F)$ to $(\R^d, \B(\R^d))$).
		Therefore, the support of its distribution
		\[
			\supp\mathrm{P}_{\bm\phi(X)}=\{ x\in\R^d \mid \mathrm{P}(\bm\phi(X)\in A)>0\
			\text{for an arbitrary open neighborhood $A$ of $x$} \}
		\]
		coincides with the smallest closed set $B$ satisfying $\mathrm{P}(\bm\phi(X)\in B)=1$.
	\end{rem}
	
	We introduce
	a technique called {\it Carath\'{e}odory-Tchakaloff subsampling} \cite{pia17},
	which reduces the number of points in a discrete measure.
	This is briefly explained below, where the argument is limited to probability measures.
	\begin{algo}{(Carath\'{e}odory-Tchakaloff subsampling)}
	\label{algo:c-t}
	Here, we explain a method to obtain
	a compressed discrete measure $\hat\mu_\mathrm{CT}$
	from a given discrete measure $\hat\mu$.
	\begin{itemize}
		\item[(1)]
			We have some discrete probability
			measure $\hat\mu$ on some space $\X$,
			which can be written as
			\[
				\hat\mu=\sum_{x\in F} a_x \delta_{x}
			\]
			for some finite set $F\subset \X$ and $a_x>0$.
			A typical case is that the measure $\hat\mu$
			is already an approximation
			of some non-discrete measure.
		\item[(2)]
			Consider some test functions $\phi_1,\ldots,\phi_d
			\in L^1(X, \hat\mu)$
			that we want to integrate.
			For the purposes of simplicity, we assume $\phi_1\equiv 1$.
			In the case $\X$ is Euclidean space, the natural choice of $\phi_i$ is a polynomial.
			
		\item[(3)]
			If $|F|$ is much larger than $d$,
			we can reconstruct a discrete measure
			\[
				\hat\mu_\mathrm{CT}=\sum_{x\in G} b_x\delta_x
			\]
			with $G\subset F$, $|G|\le d$, and $b_x>0$,
			the existence of which is assured by Theorem \ref{thm:gen-tchakaloff}.
			The obtained measure $\hat\mu_\mathrm{CT}$
			is equal to $\hat\mu$ in integrating
			test functions $\phi_1,\ldots,\phi_d$,
			so it can be regarded as a compression of $\hat\mu$.
	\end{itemize}
	\end{algo}
	\acite{pia17} gives a brief survey on this method.
	This resampling method has
	recently been developed in different backgrounds,
	such as numerical quadrature and stochastic analysis
	\cite{huy09,lit12,ryu15,som15,bit16}.
	To use this method,
	we must constructively obtain $\hat\mu_\mathrm{CT}$.
	Indeed, this measure is obtained as a basic feasible solution
	of the following LP problem with a trivial
	objective function:
	\[
		\begin{array}{rl}
			\text{minimize} & 0 \\
			\text{subject to} & A\bm{z}=\bm{b},\ \bm{z}\ge\bm0,
		\end{array}
	\]
	where $A\in\R^{d\times F}$ and $\bm{b}\in\R^{d}$
	are defined as
	\[
		A:=\left[
			\bm\phi(x)
		\right]_{x\in F}
		\qquad
		\bm{b}:=\int_\X \bm\phi(x) \dd\hat\mu(x),
		\qquad
		\bm\phi:=(\phi_1,\ldots,\phi_d)^\top.
	\]
	\acite{tch15} compares the simplex method and
	the singular-value decomposition method
	in solving this problem.
	Note that this Carath\'{e}odory-Tchakaloff subsampling
	scheme is applicable to the reduction of points in (classical)
	cubature formulas, as in Proposition \ref{prop:classic-cubature}.
	
	\begin{rem}\label{rem:CT}
		During the process of the above subsampling,
		we do not exploit the entire information of the measure $\hat\mu$.
		Rather, we only use the set $F$ and the expectation vector $\bm{b}$.
		Therefore, the above method is also applicable to the problem outlined below.
		\begin{quote}
			Given a finite set $F$,
			a vector-valued function $\bm\phi:\X\to\R^d$, and a vector $\bm{b}\in\R^d$,
			find a discrete measure $\hat\mu_\mathrm{CT}=\sum_{x\in G}b_x\delta_x$
			such that
			\begin{itemize}
				\item
					$G$ is a subset of $F$ with at most $d$ elements;
				\item
					$\int_\X \bm\phi(x)\dd\hat\mu_\mathrm{CT}(x)=\bm{b}$ holds.
			\end{itemize}
		\end{quote}
		In this formulation, the problem is feasible if, and only if,
		$\bm{b}$ is contained in the convex hull of the set
		$\{\bm\phi(x)\mid x\in F\}$.
	\end{rem}
	
	In this paper, we are primarily interested in the construction of cubature formulas
	when the larger discrete measure $\hat\mu$ is not given.
	Accordingly, we also use the terminology "Carath\'{e}odory-Tchakaloff subsampling"
	in the sense of Remark \ref{rem:CT}.
	The precise explanation of our problem formulation is given in
	Section \ref{sec:3}.
	
	
	\subsection{Preliminaries from discrete geometry}\label{sec:discrete}
	
	Let us introduce some useful notions and assertions in discrete geometry
	(see \acite{mat02} for details).
	For $S\subset\R^d$,
	the convex hull of $S$ (denoted by $\cv S$) is defined as
	\[
		\cv S:=\left\{\sum_{i=1}^m w_ix_i \lmid
		m\ge1,\ w_i>0,\ \sum_{i=1}^mw_i=1,\ x_i\in S\right\}.
	\]
	The following theorem is a well-known result originally outlined by \acite{car07}.
	\begin{thm}{\rm(Carath\'{e}odory's Theorem)}\label{thm:car}
		If $S\subset\R^d$ and $x\in\cv S$,
		then $x\in\cv T$ for some $T\subset S$ with $|S|\le d+1$.
	\end{thm}
	According to this assertion, Tchakaloff's theorem
	(Theorem \ref{thm:gen-tchakaloff}) is essentially a straightforward consequence.
	Indeed, we can show that
	$\E{\bm\phi(X)} \in \cv\supp\mathrm{P}_{\bm\phi(X)}$
	under the same condition as in Theorem \ref{thm:gen-tchakaloff},
	and there exists $x_1,\ldots, x_n\in\X$ such that $n\le d+1$ and
	$\bm\phi(x_1),\ldots,\bm\phi(x_n)\in\supp\mathrm{P}_{\bm\phi(X)}$.
	As to whether $x\in \cv T$ holds is invariant
	under affine transformations,
	if we additionally assume
	$1\in \mathop\mathrm{span}\{\phi_1,\ldots,\phi_d\}$,
	then we can reduce the number of dimensions in Theorem \ref{thm:car}
	and, accordingly, a cubature formula can be obtained with $n\le d$.

	Consider point sets in $\R^d$.
	For a set $S\subset \R^d$, the {\it affine hull} of $S$
	is defined as
	\[
		\aff S:=\left\{
			\sum_{i=1}^m w_ix_i \lmid
			m\ge1,\ w_i\in\R,\ \sum_{i=1}^m w_i=1,\ x_i\in S
		\right\}.
	\]
	Notice that $\aff S$ is the smallest affine subspace
	(shifted linear subspace) including $S$ (and $\cv S$).
	Using the notion of affine hull,
	we can define the {\it relative interior} of $S\subset\R^d$ as
	\[
		\ri S
		:=\left\{
			x\in S \lmid
			\exists\,\ve>0\ \text{s.t.}\
			y\in\aff S, \|y-x\|<\ve \Longrightarrow y\in S
		\right\},
	\]
	which is the interior of $S\subset\aff S$
	under the relative topology of $\aff S$.
	The theorem outlined below is a generalization of Carath\'{e}odory's Theorem
	using the notion of relative interior.
	\begin{thm}{\rm\cite{ste16,bon63}}\label{thm:ri}
		Let a set $S\subset\R^d$
		satisfy that $\aff S$ is a $k$-dimensional affine subspace of $\R^d$.
		Then,
		for arbitrary point $x\in\ri\cv S$,
		there exists some $T\subset S$ that satisfies
		$\aff T=\aff S$, $x\in\ri\cv T$ and $|T|\le 2k$.
	\end{thm}

	The well-known result outlined below is useful.
	It can be found in standard textbooks on discrete geometry, convex analysis, and functional analysis.
	\begin{thm}{\rm(separation theorem)}
		Let $A$ and $B$ be convex subsets of $\R^d$ satisfying $A\cap B=\emptyset$.
		Then, there exists a hyperplane $H$ such that $A$ and $B$ are included
		in different closed-half spaces defined by $H$.
		In other words,
		there exists a unit vector $c\in\R^d$ and a real number $z\in\R$ such that
		$c^\top x \ge z$ holds for all $x\in A$ and
		$c^\top y \le z$ holds for all $y\in B$.
	\end{thm}

\section{Cubature Construction Problems}\label{sec:3}
	
	\subsection{Problem setting}
	
	In this study, we consider two different problems,
	both of which have a random variable $X$ taking values in $\X$
	and a vector-valued function $\bm\phi:\X\to\R^d$
	with the condition $\E{\|\bm\phi(X)\|}<\infty$.
	For simplicity, we additionally assume that the first element of $\phi$
	is identically $1$.
	For both problems,
	we assume we can sample i.i.d. copies $X_1,X_2,\ldots$ of $X$.
	\begin{itemize}
		\item[(P1)]
			\textbf{Exact cubature problem}:
			Assuming we can calculate the exact value of $\E{\bm\phi(X)}$,
			find $n$ $(\le d)$ points $x_1,\ldots,x_n\in\X$
			and weights $w_1,\ldots,w_n>0$
			such that
			\[
				\bm\phi(x_1),\ldots,\bm\phi(x_n)
				\in\supp\mathrm{P}_{\bm\phi(X)},\qquad
				\E{\bm\phi(X)}
				=\sum_{i=1}^n w_i\phi(x_i).
			\]
		\item[(P2)]
			\textbf{Approximate cubature problem}:
			Without any knowledge of the exact value of $\E{\bm\phi(X)}$
			(except for $\E{\phi_1(X)}=1$),
			find $n$ $(\le d)$ points $x_1,\ldots,x_n\in\X$
			and weights $w_1,\ldots,w_n>0$
			such that
			\[
				\bm\phi(x_1),\ldots,\bm\phi(x_n)
				\in\supp\mathrm{P}_{\bm\phi(X)},\qquad
				\E{\bm\phi(X)}
				\simeq\sum_{i=1}^n w_i\phi(x_i).
			\]
	\end{itemize}
	(P1) is the usual (generalized) cubature construction problem.
	In (P2), since we do not know $\E{\bm\phi(X)}$,
	it is almost impossible to find an exact cubature formula.
	However, Monte Carlo or QMC
	integration well-approximates the expectation $\E{\bm\phi(X)}$
	without requiring any prior knowledge of said expectation.
	Therefore, it is possible to construct
	an approximate cubature formula.
	
	\subsection{Monte Carlo approach to the exact cubature problem}\label{sec:3-1}
	To solve (P1),
	we can use the i.i.d. samples $X_1,X_2,\ldots$ as candidates of
	the points used in the cubature formula.
	If, for some $N$, we have
	\[
		\E{\bm\phi(X)}\in\cv\{\bm\phi(X_1), \ldots, \bm\phi(X_N)\},
	\]
	then we can construct a cubature formula
	using Carath\'{e}odory-Tchakaloff subsampling (Method \ref{algo:c-t}).
	
	Indeed, we can prove that
	said $N$ exists almost surely and $\E{\min N}<\infty$
	($\min N$) is Borel measurable, because
	$\{(x_1,\ldots, x_n)\in \R^{d\times n}
	\mid y\not\in\cv\{x_1,\ldots, x_n\}\}$
	is an open set of $\R^{d\times n}$
	for each $y\in\R^d$).
	This fact can be proved using the lemma outlined below.
	\begin{lem}\label{lem:exp-ri}
		The expectation $\E{\bm\phi(X)}$ belongs to the relative interior of
		$\cv\supp\mathrm{P}_{\bm\phi(X)}$.
	\end{lem}

	\begin{proof}
		Denote $\supp\mathrm{P}_{\bm\phi(X)}$ by $S$.
		First, we show that $\E{\bm\phi(X)}\in\aff S$.
		If we suppose not,
		we can take a nonzero vector $c\in \R^d$ such that
		$c^\top x$ is constant for $x\in S$ and $c^\top \E{\bm\phi(X)}$ is a different value.
		Indeed, this would be an absurd assumption to make, since $c^\top \E{\bm\phi(X)}$ is the expectation
		of $c^\top\bm\phi(X)$, which should be equal to the aforementioned constant.
		Therefore, we have $\E{\bm\phi(X)}\in\aff S$.
		
		Suppose $\E{\bm\phi(X)}\not\in\ri\cv S$.
		By translating if necessary,
		we can assume $\E{\bm\phi(X)}=0$
		and $\aff S$ is a linear subspace of $\R^d$.
		Then, for each $n=1,2,\ldots$,
		$A_n:=\{ x\in \aff S \mid \|x\|<1/n\}$
		has a nonempty intersection with $(\cv S)^c$.
		Take a sequence $a_1,a_2,\ldots$ such that
		$a_n\in A_n\setminus \cv S$ for each $n$.
		Then, by the separation theorem,
		we have a sequence of unit vectors $c_n\in\aff S$ that separate $a_n$ from $\cv S$;
		i.e., they satisfy $c_n^\top a_n \ge c_n^\top x$ for all $x\in\cv S$.
		Here, by the compactness of the set $C:=\{c\in \aff S \mid \|c\|=1\}$,
		the sequence $c_1,c_2,\ldots$ has a subsequence that is convergent with $c\in C$.
		Because $c_n^\top a_n \le \|c_n\|\cdot \|a_n\|<1/n$ for each $n$,
		we have $c^\top x \le 0$ for all $x\in \cv S$.
		As we have assumed $\E{\bm\phi(X)}=0$,
		we can obtain $c^\top\bm\phi(X)=0$ almost surely,
		so $S=\supp\mathrm{P}_{\bm\phi(X)}$ is included in the lower-dimensional subspace
		$\{x\in \aff S \mid c^\top x=0\}$ (see Remark \ref{rem:rigorous}).
		This obviously contradicts the definition of $\aff S$ (the smallest affine subspace including $S$).
		Therefore, we have $\E{\bm\phi(X)}\in \ri\cv\supp\mathrm{P}_{\bm\phi(X)}$.
	\end{proof}

	The theorem outlined below, which shows the existence of desired $N$, is the main result of this paper.

	\begin{thm}\label{thm:main}
		Let $X_1, X_2, \ldots$ be i.i.d. copies of $X$.
		Then, with probability $1$,
		there exists a positive integer $N$ such that
		$\E{\bm\phi(X)}$ is contained in $\cv\{\bm\phi(X_1),\ldots,\bm\phi(X_N)\}$.
	\end{thm}
	\begin{proof}
		Denote $\supp\mathrm{P}_{\bm\phi(X)}$ by $S$.
		By Lemma \ref{lem:exp-ri}
		and Theorem \ref{thm:ri},
		there exists a set $T\subset S$
		such that $\E{\bm\phi(X)}\in \ri\cv T$
		and $|T|\le 2d$.
		Then, there exists an $r>0$ satisfying
		\[
			B_r:=\{x\in \aff S
			\mid \|x-\E{\bm\phi(X)}\|\le r\}\subset \cv T.
		\]
		If the elements of $T$ are subscripted as
		$T=\{x_1,\ldots, x_{|T|}\}$,
		we can regard $\cv T$ as the image of the mapping
		\[
			\tau:\Delta^{|T|}\to \aff S;\qquad
			(t_1, \ldots, t_{|T|})\mapsto t_1x_1+\cdots+t_{|T|}x_{|T|},
		\]
		where $\Delta^{|T|}:=\{(t_1,\ldots,t_{|T|})\in\R^{|T|}\mid
		t_1,\ldots,t_{|T|}\ge0,\ t_1+\cdots+t_{|T|}=1\}$.
		
		We prove that,
		if points $y_1, \ldots, y_{|T|}\in S$ satisfy
		$\|y_i-x_i\|<r$ for each $i=1,\ldots, T$,
		then $\E{\bm\phi(X)}\in \cv\{y_1,\ldots, y_{|T|}\}$ holds.
		If this is not true, there exist the points $y_1,\ldots, y_{|T|}$
		with $\E{\bm\phi(X)}\not\in\cv\{y_1,\ldots,y_{|T|}\}$.
		Then, by the separation theorem,
		there exists a hyperplane $H\subset \aff S$
		going through $\E{\bm\phi(X)}$
		such that all the points $y_1,\ldots,y_{|T|}$ are contained in
		one closed-half space (of $\aff S$) made by $H$.
		We now can take a point $z\in B_r$ such that
		$\|z-\E{\bm\phi(X)}\|=r$, $z-\E{\bm\phi(X)}\perp H$ and
		$z$ is lying on the other side of $H$ than $y_1,\ldots,y_{|T|}$.
		Since $B_r\subset\cv T$,
		we can take a weight $(t_1,\ldots,t_{|T|})\in \tau^{-1}(z)$ and, accordingly,
		we have
		\[
			\|(t_1y_1+\cdots+t_{|T|}y_{|T|})-z\|
			\le t_1\|y_1-x_1\|+\cdots+t_{|T|}\|y_{|T|}-x_{|T|}\|
			< r.
		\]
		This means that $t_1y_1+\cdots+t_{|T|}y_{|T|}$ lies on the other side
		of $H$ than $y_1,\ldots,y_{|T|}$, which is impossible.
		Therefore, the aforementioned assertion is true.
		
		By the definition of $\supp\mathrm{P}_{\bm\phi(X)}$,
		we have $\mathrm{P}(\bm\phi(X)\in S,\ \|\bm\phi(X)-x_i\|<r)>0$
		for each $i=1,\ldots, |T|$.
		Thus, the assertion of the theorem follows.
	\end{proof}
	As mentioned above, we can also see from the above proof that
	$\E{\min N}<\infty$.
	There exists a concept called the (random) degree function
that generalizes $\min N$:
	\[
		\deg(x; \mathrm{P}_X):=\min\{n \mid x\in \cv \{X_1, \ldots, X_n\}\},
	\]
	where $x\in\R^d$, $\mathrm{P}_X$ is a probability law over $\R^d$
	and $X_1,X_2,\ldots$ is an i.i.d. sequence following
	the law $\mathrm{P}_X$.
	This function and the depth, which is related to the degree,
	have been treated by \citeasnoun{cas07} and \citeasnoun{liu90}.
	Our interest is $\deg(\E{X}; \mathrm{P}_X)$,
	but existing results are based on assuming an absolute continuity
	(and angular symmetry) of the law $\mathrm{P}_X$.
	Indeed, these are strong assumptions; however,
	the results for depths functions may still be useful in analyzing our method
	for concrete distributions in future work.

	\subsection{Approaches to the approximate cubature problem}\label{sec:3-2}
	Considering (P2),
	there are two cases, which are outlined below.
	\begin{itemize}
		\item[(a)]
			We are familiar with the distribution of $X$,
			but we do not know the exact value of $\E{\bm\phi(X)}$.
		\item[(b)]
			We are not familiar with $X$.
	\end{itemize}

	In (a),
	we assume there exists known QMC formulas
	or some general discrete approximation of $\mathrm{P}_X$.
	In this case,
	if $\bm\phi$ is a good basis in some sense
	(e.g., the space of functions we want to integrate),
	it is not necessary to use all of the information of existing formulas.
	Mathematically speaking,
	from a given discrete approximation $\hat\mu$ of $\mathrm{P}_X$,
	we can make a Carath\'{e}odory-Tchakaloff subsampling
	$\hat\mu_\mathrm{CT}$
	of $\hat\mu$.
	This is one of the usual applications of the subsampling method.
	
	In (b), however, we do not have such a measure in advance.
	Rather, in such a situation, we generally have to carry out
	Monte Carlo integration every time we want the integration of some integrand.
	However, once we have a sufficiently large i.i.d. sample
	$X_1,X_2,\ldots,X_N$ of $X$,
	we can construct a subsampling $\hat\mu_\mathrm{CT}$ of the measure
	$\hat\mu=\frac1N\sum_{i=1}^N\delta_{X_i}$.
	This $\hat\mu_\mathrm{CT}$ is random but statistically satisfies
	\[
		\E{\int_\X\bm\phi(x)\dd\hat\mu_\mathrm{CT}(x)}
		=\E{\int_\X\bm\phi(x)\dd\hat\mu(x)}
		=\E{\bm\phi(X)},
	\]
	\[
		\E{\left|\int_\X\bm\phi(x)\dd\hat\mu_\mathrm{CT}(x)
		-\E{\bm\phi(X)}
		\right|^2}
		=\E{\left|\int_\X\bm\phi(x)\dd\hat\mu(x)
			-\E{\bm\phi(X)}
			\right|^2}
		=\mathrm{O}\left(\frac1N\right)
	\]
	under an additional moment condition $\E{\|\bm\phi(X)\|^2}<\infty$.
	The merit of constructing $\hat\mu_{\mathrm{CT}}$ is that
	we only need the value at $d$ points for each integrand,
	whereas the standard Monte Carlo requires $N$ points valuation of each integrand.
	Therefore, the Monte Carlo approach to (P2) is also effective.

\section{Experimental Results}\label{sec:4}
	
	To identify how many samples we should take in constructing cubature formulas,
	we conducted numerical experiments.
	For each pair $(s, m) \in \{1,2,3,4,5\}\times\{1,2,3,4,5\}$,
	we estimated the smallest value of $N$ with
	\begin{align}
		\mathrm{P}\left(
			\E{\bm\phi(X)}\in\cv\{\bm\phi(X_1),\ldots,\bm\phi(X_N)\}
		\right)\ge \frac12,
		\label{eq:4-1}
	\end{align}
	where
	$\bm\phi$ is a vector-valued function, the entries of which are composed of
	all $s$-variate monomials with degree $m$ at most,
	and $X$ is a uniform random variable on $[0, 1)^s$, and $X_1, X_2,\ldots$ are i.i.d. copies of $X$.
	In judging whether $\E{\bm\phi(X)}\in\cv\{\bm\phi(X_1),\ldots,\bm\phi(X_n)\}$ holds
	given the values of $X_1,\ldots,X_n$,
	we used the LP solver {\ttfamily glpsol} in the package GNU LP Kit\footnote{\url{https://www.gnu.org/software/glpk/}} (GLPK, version 4.65).
	Note that we do not have to construct a cubature explicitly
	in confirming $\E{\bm\phi(X)}$ is contained in the convex hull of given vectors.
	Although
	we used the simplex method in the experiment and actually constructed a cubature formula with the desired size,
	in this section we can use other means, such as the interior-point method.
	To estimate the smallest $N$ satisfying Eq. (\ref{eq:4-1}), we conducted binary search on $[d_{s, m}, 10000]$,
	where $d_{s, m}$ is the dimension of the vector $\bm\phi$.
	During the binary search (with $n\in [d_{s, m}, 10000]$),
	we independently sample $(X^i_1,\ldots,X^i_n)$ 20 times ($i=1,\ldots, 20$); when there were at least $10$ indices $i\in\{1,\ldots,20\}$
	such that
	$\E{\bm\phi(X)}\in\cv\{\bm\phi(X^i_1),\ldots,\bm\phi(X^i_n)\}$ held,
	$n$ was judged to be larger than (or equal to) $N$;
	otherwise it was judged as being smaller than $N$.
	The experimental results are shown in Table \ref{table:1}.
	Unfortunately, they are not the exact values of $N$
	as statistical and numerical errors exist (due to the precision of {\ttfamily double}); however,
	they work as estimated values of $N$.
	\begin{table}[ht]
	\caption{\small Estimation of $N$ for each $(s, m)$.
		In each entry, ``$A\ (B)$" means $(\text{estimated $N$}) = A$ and $d_{s, m}=B$.
		We had to set a 10-second time limit for the LP solver in the case $(s, m)=(5, 5)$ because, at times,
		the solver stopped due to numerical instability }\label{table:1}
		\begin{center}
			\begin{tabular}{|c||c|c|c|c|c|}
				\hline
				$s \setminus m$ & 1 & 2 & 3 & 4 & 5 \\
				\hhline{|=#=|=|=|=|=|}
				1 & 3 (2) & 6 (3) & 9 (4) & 12 (5) & 14 (6) \\
				\hline
				2 & 5 (3) & 10 (6) & 22 (10) & 36 (15) & 66 (21) \\
				\hline
				3 & 6 (4) & 21 (10) & 49 (20) & 90 (35) & 160 (56) \\
				\hline
				4 & 7 (5) & 33 (15) & 81 (35) & 179 (70) & 338 (126) \\
				\hline
				5 & 11 (6) & 44 (21) & 123 (56) & 311 (126) & 663 (252) \\
				\hline
			\end{tabular}
		\end{center}
	\end{table}

	Although it might be difficult to give theoretical bounds for $N$,
	from the experiment it is event that $N$ should not be as large as it is compared with $d_{s, m}$.
	In this polynomial case,
	$N$ takes values equivalent to roughly $2d_{s, m}$ or $3d_{s, m}$.
	Accordingly, we can obtain a cubature formula with realistic computational costs,
	as long as $d_{s, m}$, which is the number of test functions, is not so large.
	
\section{Concluding Remarks}\label{sec:5}

	In this study, we have shown that we can construct a general cubature formula
	of a probability measure
	if we have an i.i.d. sampler from the measure and if we know the exact mean values
	of the given test functions.
	We can construct a cubature formula by solving an LP problem
	given sufficiently large i.i.d. sampling.
	The proof of this fact is based on the discrete geometry results.
	We have also seen that our Monte Carlo approach is effective even
	if we do not know the mean values of the test functions.
	Indeed, we can construct an approximate cubature formula
	using points given by the Monte Carlo sampling,
	a process that can be regarded as a data compression
	(Carath\'{e}odory-Tchakaloff subsampling).
	By numerical experiments,
	we have empirically seen that
	the size of the LP problem we were required to solve should not be so large
	compared with the number of test functions.
	
	Although we assumed that we can sample the exact i.i.d. sequence from
	the given probability measure, in future work
	we should actually consider
	the possibility of not having exact samplers
	(typically the case we use MCMC to sample).

\section*{Acknowledgments}

The author is grateful to Ken'ichiro Tanaka for his valuable suggestions.
The author wold like to thank Enago (www.enago.jp) for the English language review.
\bibliographystyle{dcu}
\bibliography{cite1}

%
%

\end{document}